\documentclass[11pt, letterpaper,reqno]{amsart}

\usepackage{amsmath}
\usepackage{amssymb}
\usepackage{amsfonts}
\usepackage{ dsfont }
\usepackage{cases}
\usepackage{cite}
\usepackage[hyphens]{url}
\usepackage{hyperref}
\usepackage{amsthm}
\usepackage{upgreek}
\usepackage{color}

\usepackage{xcolor}
\definecolor{color0}{gray}{.50}
\definecolor{color1}{rgb}{0,.2,.8}
\definecolor{color2}{rgb}{1,.2,0}
\definecolor{color3}{rgb}{.8,.5,1}

\numberwithin{equation}{section}

\newtheorem{theorem}{Theorem}[section]

\newtheorem{lemma}[theorem]{Lemma}
\newtheorem{proposition}[theorem]{Proposition}
\newtheorem{definition}[theorem]{Definition}
\newtheorem{assumption}[theorem]{Assumption}
\newtheorem{remark}[theorem]{Remark}
\newtheorem{example}[theorem]{Example}

\newcommand{\leref}{Lemma~\ref}

\newcommand{\prref}{Proposition~\ref}
\newcommand{\thref}{Theorem~\ref}

\newcommand{\asref}{Assumption~\ref}
\newcommand{\Rho}{{\pmb\rho}}
\newcommand{\Tau}{{\pmb\tau}}
\newcommand{\Sig}{{\pmb\sigma}}
\newcommand{\E}{\mathbb{E}}
\newcommand{\PP}{\mathbb{P}}

\newcommand{\R}{\mathbb{R}}
\newcommand{\eps}{\epsilon}

\newcommand{\cT}{\mathcal{T}}
\newcommand{\T}{\mathbb{T}}

\newcommand{\cF}{\mathcal{F}}

\title[]{Multi-player stopping games in continuous time}
\author[]{Zhou Zhou}
\address{Institute for Mathematics and its Applications, University of Minnesota}
\email{zhouzhou@ima.umn.edu}
\date{\today}

\begin{document}
\maketitle
\begin{abstract}
We consider multi-player stopping games in continuous time. Unlike Dynkin games, in our games the payoff of each player is revealed after all the players stop. Moreover, each player can adjust her own stopping strategy by observing other players' behaviors. Assuming the continuity of the payoff functions in time, we show that there always exists an $\eps$-Nash equilibrium in pure stopping strategies for any $\eps>0$.
\end{abstract}

\section{Introduction}
Since\cite{Dynkin}, Dynkin game has attracted a lot of research. We refer to \cite{Dynkin,Zhang3,Solan,Hamadene,Kifer,Solan1,Solan2,Ko3,Lepeltier,Solan3,Solan4,Touzi,Neveu,Bismut,Ferenstein} and the references therein. In a Dynkin game, each player chooses a stopping strategy, and the payoffs of the players are revealed as long as one player stops. In other words, the game ends at the minimum of the stopping strategies. With some assumptions on the relationship between the payoff processes, it is proved in e.g., \cite {Zhang3,Solan} that a two-player non-zero-sum Dynkin game in continuous time admits a Nash equilibrium in pure stopping times. In general, without such assumptions, the two-player Dynkin game only has a Nash equilibrium in randomized strategies, see e.g., \cite{Solan}. It is known that when there are more than two players, the Dynkin game in continuous time may not have any Nash Equilibrium in randomized strategies even if the payoff processes are constant (see e.g., \cite{Solan4}).

As a classical model of stopping games, Dynkin game has many applications in economics and finance. However, it has two major limitations in many situations. First of all, in practice, it is more often that, even if a player has made the decision first, her payoff can still be affected by other players' decisions later on. In other words, it is more reasonable that the games end at the maximum of the stopping strategies. Second, a wise player would adjust her strategy after she observes other players' actions, and Dynkin games cannot incorporate this ``game'' feature.

Recently \cite{ZZ6,ZZ7,ZZ9,ZZ10} begin to consider the stopping games with these two features. In particular, assuming that the payoff functions are continuous in time, \cite{ZZ9} shows that a two-player non-zero-sum stopping game always admits an $\eps$-Nash equilibrium in pure strategies for any $\eps>0$.

In this paper, we extend the result in \cite{ZZ9} to the multi-player case. To be more specific, given a filtered probability space $(\Omega,\cF,(\cF)_{t\in[0,\infty]},\PP)$, we consider the stopping game in continuous time
$$u^i(\Rho^1,\dotso,\Rho^N)=\E[U^i(\Rho^1,\dotso,\Rho^N)],\quad i=1,\dotso,N,$$
where the player $i$ chooses $\Rho^i$ to maximize the payoff $u^i$. Here $U^i(t_1,\dotso,t_N)$ is $\cF_{t_1\vee\dotso\vee t_N}$-measurable instead of $\cF_{t_1\wedge\dotso\wedge t_N}$-measurable as is assumed in Dynkin games. That is, the game ends at the maximum of players' stopping. Moreover, here $\Rho^i$ is not a stopping times. It is a strategy that can be adjusted according to other players' actions. By assuming $U^i$ is uniformly continuous in $(t_1,\dotso,t_N)$, we show that the game always admits an $\eps$-Nash equilibrium in pure strategies for any $\eps>0$.

We prove the result by an induction on the number of the players. That is, we construct an $\eps$-Nash equilibrium of the $N$-player game from the $\eps$-Nash equilibriums of $(N-1)$-player games as well as $\eps$-saddle points of some zero-sum games. To reduce the burden of the notation, we only focus on the three-player case ($N=3$), and the proof still works accordingly for  the case with more players.

Our game has a wide range of applications, e.g., when companies choose times to take actions, and e.g., when investors who both short and long American options try to maximize their utilities.
   
The paper is organized as follows. In the next section, we provide the setup and the main result. In Section 3, we provide some auxiliary results. In section 4, we use these auxiliary results to construct an $\eps$-equilibrium for the original game.

\section{Setup and the main result}

In this section, we will provide the general setup and the main result. \thref{t1} is the main result of this paper.

Let $(\Omega,\cF,(\mathcal{F}_t)_{t\in[0,\infty]},\PP)$ be a filtered probability space, where $\mathcal{F}=\mathcal{F}_\infty$ and $(\mathcal{F}_t)_{t\in[0,\infty]}$ satisfies the usual conditions. To avoid the technical difficulties stemming from the verification of path regularities of some related processes, we assume that $\Omega$ is at most countable, and $\PP$ is supported on $\Omega$. Let $\cT$ be the set of stopping times. For $\rho\in\cT$, denote $\cT_\rho$ (resp. $\cT_{\rho+}$) be the set of stopping times that is no less (resp. strictly greater) than $\rho$ on $\{\rho<\infty\}$.

\begin{definition}
Let $N\in\mathbb{N}$ and $I$ be the set of all the subsets of $\{1,\dotso,N-1\}$. A $2^{N-1}$-tuple $\Rho=(\rho_\alpha)_{\alpha\in I}$ is said to be a stopping strategy (of order $N$), if $\rho:=\rho_\emptyset\in\cT$, and for any $I=(i_1,\dotso,i_n)\subset\{1,\dotso,N-1\}$ with $i_1<\dotso<i_n$,
$$\rho_{i_1,\dotso,i_n}:[0,\infty]^n\times\Omega\mapsto[0,\infty]\ \text{is}\ \mathcal{B}([0,\infty]^n)\otimes\cF\text{-measurable},$$
and
$$\rho_{i_1,\dotso,i_n}(t_1,\dotso,t_n,\cdot)\in\cT_{(t_1\vee\dotso\vee t_n)+}.$$
Denote $\T^N$ as the set of stopping strategies of order $N$. For $\sigma\in\cT$, let
$$\T_\sigma^N:=\{\Rho=(\rho_\alpha)_{\alpha\in I}\in\T^N:\ \rho\geq\sigma\}.$$
\end{definition}

The interpretation of $\Rho\in\T^N$ is as follows. Suppose there are $N$ players, and each of them  needs to choose a time to make a decision (stop). Let $\Rho$ be player $N$'s stopping strategy. At the beginning, player $N$ chooses an initial stopping time $\rho$. If no other players stop before $\rho$, then player $N$ stops at time $\rho$. Otherwise, player $N$ observes someone stops before $\rho$. Say, it is player 1 who stops first at time $t<\rho$. Then player $N$ observes player 1's action, and immediate switches to strategy $\rho_1(t)$. In general, $\rho_{i_1,\dotso,i_n}(t_1,\dotso,t_n,\cdot)$ represents the strategy that player $N$ uses, if she has not stopped by time $t_1\vee\dotso\vee t_n$, and she observes that players $i_1,\dotso,i_n$ have stopped at time $t_1,\dotso,t_n$ respectively.

Let $\Rho^i\in\T^N$, $i=1,\dotso,N$, which presents the stopping strategy for player $i$ in an $N$-player game. Denote $\Rho^i[\Rho^{-i}]$ as the actual time when player $i$ stops under the other players' stopping strategies $\Rho^{-i}:=(\Rho^1,\dotso,\Rho^{i-1},\Rho^{i+1},\dotso,\Rho^N)$. Due  to the complexity of the notation, we will not explicitly write out the expression of $\Rho^i[\Rho^{-i}]$. Instead, we give two examples when $N=2$ and $N=3$.

\begin{example}
Let $\Rho=(\rho,\rho_1),\Tau=(\tau,\tau_1)\in\T^2$. Then
$$\Rho[\tau]=\rho 1_{\{\rho\leq\tau\}}+\rho_1(\tau) 1_{\{\rho>\tau\}},$$
where $\rho_1(\tau):=\rho_1(\tau(\cdot),\cdot)$.
\end{example}

\begin{example}
Let $\Rho=(\rho,\rho_2,\rho_3,\rho_{23}),\Tau=(\tau,\tau_1,\tau_3,\tau_{13}),\Sig=(\sigma,\sigma_1,\sigma_2,\sigma_{12})\in\T^3$. Then
\begin{eqnarray}
\notag \Rho[\Tau,\Sig]&=&\rho 1_{\{\rho\leq\tau\wedge\sigma\}}+\rho_{23}(\tau,\sigma) 1_{\{\rho>\tau=\sigma\}}\\
\notag &+&\rho_3(\sigma) 1_{\{\rho,\tau>\sigma\}\cap\{\rho_3(\sigma)\leq\tau_3(\sigma)\}}+\rho_{23}(\tau_3(\sigma),\sigma) 1_{\{\rho,\tau>\sigma\}\cap\{\rho_3(\sigma)>\tau_3(\sigma)\}}\\
\notag &+&\rho_2(\tau) 1_{\{\rho,\sigma>\tau\}\cap\{\rho_2(\tau)\leq\sigma_2(\tau)\}}+\rho_{23}(\tau,\sigma_2(\tau)) 1_{\{\rho,\sigma>\tau\}\cap\{\rho_2(\tau)>\sigma_2(\tau)\}},
\end{eqnarray}
where $\rho_{23}(\tau,\sigma):=\rho_{23}(\tau(\cdot),\sigma(\cdot),\cdot)$.
\end{example}

For $i=1,\dotso,N$, let $U^i:[0,\infty]^N\times\Omega\mapsto\R$ such that $U^i(t_1,\dotso,t_N)$ is $\cF_{t_1\vee\dotso\vee t_N}$-measurable. We make the following standing assumption throughout this paper.
\begin{assumption}\label{a1}
There exists a bounded nondecreasing function $\eta:\R_+\mapsto\R_+$ satisfying $\lim_{\delta\searrow 0}\eta(\delta)=\eta(0)=0$, such that for $i=1,\dotso,N$, and any $(t_1,\dotso,t_N),(t_1',\dotso,t_N')\in[0,\infty]^N$,
$$|U^i(t_1,\dotso,t_N)-U^i(t_1'\dotso,t_N')|<\eta(|t_1-t_1'|+\dotso+|t_N-t_N'|).$$
\end{assumption}

Now for any $\theta\in\cT$, consider the $N$-player stopping game in continuous time
\begin{equation}\label{e1}
u^i(\Rho^1,\dotso,\Rho^N):=\E_\theta\left[U^i(\Rho^1[\Rho^{-1}],\dotso,\Rho^N[\Rho^{-N}])\right],\quad\Rho^1,\dotso,\Rho^N\in\T_\theta^N,\quad i=1,\dotso,N,
\end{equation}
where $\E_\theta[\cdot]:=\E[\cdot|\cF_\theta]$. Here player $i$ chooses the stopping strategy $\Rho^i$ to maximize her own utility $U^i$. Recall the definition of an $\eps$-Nash equilibrium.

\begin{definition}
For $\eps>0$, the $N$-tuple $(\hat\Rho^1,\dotso,\hat\Rho^N)\in(\T_\theta^N)^N$ is said to be an $\eps$-Nash equilibrium (in pure stopping strategies) for the game \eqref{e1}, if for any $\Rho^1,\dotso,\Rho^N\in\T_\theta^N$,
$$u^i(\hat\Rho^1,\dotso\hat\Rho^{i-1},\Rho^i,\hat\Rho^{i+1},\dotso,\hat\Rho^N)\leq u^i(\hat\Rho^1,\dotso\hat\Rho^{i-1},\hat\Rho^i,\hat\Rho^{i+1},\dotso,\hat\Rho^N)+\eps,\quad i=1,\dotso,N.$$
\end{definition}

Below is the main result of this paper.

\begin{theorem}\label{t1}
Under \asref{a1}, there are exists an $\eps$-Nash equilibrium for the game \eqref{e1} for any $\eps>0$.
\end{theorem}

\begin{remark}
Our game can also be adapted to the case when each player has multiple stopping. For example, suppose each player has double stopping. Then we can treat player $i$ as new players $2i-1$ and $2i$, and let $\tilde U^{2i-1}:=\tilde U^{2i}:=U^i$. Of course, by doing so, a Nash equilibrium in the new game may not be optimal for each player in the old game. (Recall that it is possible that a Nash equilibrium $(x^*,y^*)$  for $f^i(x,y)=f(x,y),\ i=1,2$ may not be optimal for $\sup_{x,y}f(x,y)$.)
\end{remark}

We will prove \thref{t1} by an induction on the number of players. We will construct an $\eps$-Nash equilibrium of the $N$-player game from $\eps$-Nash equilibriums of $(N-1)$-player games, as well as some $\eps$-saddle points of some related zero-sum games. By \cite[Theorem 2.4]{ZZ9}, \thref{t1} holds when $N=2$. (Although in \cite{ZZ9} the game starts from time $t=0$, but the proof there still works if the game starts at any stopping time.) To reduce the burden of the notation, we will only prove the result when $N=3$. It can be seen later on that our proof also works for general $N$, as long as we assume that \thref{t1} holds for $N-1$.

We will first provide some auxiliary results in the next section.

\section{Some auxiliary results}

In this section, we provide some auxiliary results in preparation for the proof of \thref{t1} when $N=3$. Some of the proofs may admit simpler solutions, yet we demonstrate them in such ways in order to let the proofs also work for $N>3$. It is worth noting that as $\Omega$ is at most countable, those results that hold w.r.t. any $t\in[0,\infty]$ also hold w.r.t. any $\theta\in\cT$.

\begin{lemma}\label{l1}
For $i=1,2$, let $G^i:[0,\infty]^3\times\Omega\mapsto\R$ such that $G^i(r,s,t)$ is $\cF_{r\vee s\vee t}$-measurable and satisfies \asref{a1}. Then for any $\eps>0$, there exists $h>0$, such that for any $t\in[0,\infty)$, there exists $(\hat\Rho,\hat\Tau)\in(\T_t^2)^2$ such that for any $\delta\in[0,h]$, $(\hat\Rho,\hat\Tau)$ is an $\eps$-Nash equilibrium for the game
$$\E_{t-\delta}\left[G^i(\Rho[\Tau],\Tau[\Rho],t-\delta)\right],\quad\Rho,\Tau\in\T_{t-\delta}^2,\quad i=1,2.$$
\end{lemma}

\begin{proof}
For any $\eps>0$, let $h>0$ such that $\eta(h)<\eps$. Now for any $t\in[0,\infty)$, by \cite[Theorem 2.4]{ZZ9}, there exists an $\eps$-Nash equilibrium $(\Rho^*=(\rho^*,\rho_1^*),\Tau^*=(\tau^*,\tau_1^*))\in(\T_t)^2$ for the game
$$\E_t\left[G^i(\Rho[\Tau],\Tau[\Rho],t)\right],\quad\Rho,\Tau\in\T_t^2,\quad i=1,2.$$
Let $\hat\Rho:=(\rho^*,\hat\rho_1)$ and $\hat\Tau:=(\tau^*,\hat\tau_1)$, where
\[
\hat\rho_1(s) = 
\begin{cases}
\rho_1^*(t), & \text{if}\ s<t\neq\rho^*,\\
 t,           & \text{if}\ s<t=\rho^*,\\
 \rho_1^*(s), &\text{if}\ s\geq t,
\end{cases}
\quad\text{and}\quad
\hat\tau_1(s) = 
\begin{cases}
\tau_1^*(t), & \text{if}\ s<t\neq\tau^*,\\
 t,           & \text{if}\ s<t=\tau^*,\\
 \tau_1^*(s), &\text{if}\ s\geq t.
\end{cases}
\]
It can be shown that $\hat\Rho,\hat\Tau\in\T_t^2$, and for any $\Rho,\Tau\in\T_t^2$,
$$\hat\Rho[\Tau]=\Rho^*[\Tau],\quad\Tau[\hat\Rho]=\Tau[\Rho^*],\quad\hat\Tau[\Rho]=\Tau^*[\Rho],\quad\Rho[\hat\Tau]=\Rho[\Tau^*].$$
Then $(\hat\Rho,\hat\Tau)$ is a $3\eps$-Nash equilibrium for the game
$$\E_t\left[G^i(\Rho[\Tau],\Tau[\Rho],t-\delta)\right],\quad\Rho,\Tau\in\T_t^2,\quad i=1,2.$$
for any $\delta\in[0,h]$. Indeed, for any $\Rho\in\T_t^2$,
\begin{eqnarray}
\notag\E_t\left[G^1(\Rho[\hat\Tau],\hat\Tau[\Rho],t-\delta)\right]&\leq&\E_t\left[U^1(\Rho[\hat\Tau],\hat\Tau[\Rho],t)\right]+\eps\\
\notag &=&\E_t\left[G^1(\Rho[\Tau^*],\Tau^*[\Rho],t)\right]+\eps\\
\notag&\leq&\E_t\left[G^1(\Rho^*[\Tau^*],\Tau^*[\Rho^*],t)\right]+2\eps\\
\notag &=&\E_t\left[G^1(\hat\Rho[\Tau^*],\hat\Tau[\Rho^*],t)\right]+2\eps\\
\notag &=&\E_t\left[G^1(\hat\Rho[\hat\Tau],\hat\Tau[\hat\Rho],t)\right]+2\eps\\
\notag &\leq&\E_t\left[G^1(\hat\Rho[\hat\Tau],\hat\Tau[\hat\Rho],t-\delta)\right]+3\eps.
\end{eqnarray}

Now take $\Rho=(\rho,\rho_1)\in\T_{t-\delta}^2$.  We have that
\begin{eqnarray}
\notag \E_{t-\delta}\left[G^1(\Rho[\hat\Tau],\hat\Tau[\Rho],t-\delta) 1_{\{\rho<t<\tau^*\}}\right]&=&\E_{t-\delta}\left[\E_t\left[G^1(\rho,\tau_1^*(t),t-\delta)\right] 1_{\{\rho<t<\tau^*\}}\right]\\
\notag &\leq&\E_{t-\delta}\left[\E_t\left[G^1(t,\tau_1^*(t),t-\delta)\right] 1_{\{\rho<t<\tau^*\}}\right]+\eps\\
\notag &=&\E_{t-\delta}\left[\E_t\left[G^1(\bar\Rho[\hat\Tau],\hat\Tau[\bar\Rho],t-\delta)\right] 1_{\{\rho<t<\tau^*\}}\right]+\eps\\
\notag &\leq&\E_{t-\delta}\left[\E_t\left[G^1(\hat\Rho[\hat\Tau],\hat\Tau[\hat\Rho],t-\delta)\right] 1_{\{\rho<t<\tau^*\}}\right]+4\eps\\
\notag &=&\E_{t-\delta}\left[G^1(\hat\Rho[\hat\Tau],\hat\Tau[\hat\Rho],t-\delta) 1_{\{\rho<t<\tau^*\}}\right]+4\eps,
\end{eqnarray}
and
\begin{eqnarray}
\notag \E_{t-\delta}\left[G^1(\Rho[\hat\Tau],\hat\Tau[\Rho],t-\delta) 1_{\{\rho<t=\tau^*\}}\right]&=&\E_{t-\delta}\left[\E_t\left[G^1(\rho,t,t-\delta)\right] 1_{\{\rho<t=\tau^*\}}\right]\\
\notag &\leq&\E_{t-\delta}\left[\E_t\left[G^1(t,t,t-\delta)\right] 1_{\{\rho<t=\tau^*\}}\right]+\eps\\
\notag &=&\E_{t-\delta}\left[\E_t\left[G^1(\bar\Rho[\hat\Tau],\hat\Tau[\bar\Rho],t-\delta)\right] 1_{\{\rho<t=\tau^*\}}\right]+\eps\\
\notag &\leq&\E_{t-\delta}\left[\E_t\left[G^1(\hat\Rho[\hat\Tau],\hat\Tau[\hat\Rho],t-\delta)\right] 1_{\{\rho<t=\tau^*\}}\right]+4\eps\\
\notag &=&\E_{t-\delta}\left[G^1(\hat\Rho[\hat\Tau],\hat\Tau[\hat\Rho],t-\delta) 1_{\{\rho<t=\tau^*\}}\right]+4\eps,
\end{eqnarray}
and
\begin{eqnarray}
\notag \E_{t-\delta}\left[G^1(\Rho[\hat\Tau],\hat\Tau[\Rho],t-\delta) 1_{\{\rho\geq t\}}\right]&=& \E_{t-\delta}\left[\E_t\left[G^1(\Rho[\hat\Tau],\hat\Tau[\Rho],t-\delta)\right] 1_{\{\rho\geq t\}}\right]\\
\notag &\leq& \E_{t-\delta}\left[\E_t\left[G^1(\hat\Rho[\hat\Tau],\hat\Tau[\hat\Rho],t-\delta)\right] 1_{\{\rho\geq t\}}\right]+3\eps\\
\notag &=& \E_{t-\delta}\left[G^1(\hat\Rho[\hat\Tau],\hat\Tau[\hat\Rho],t-\delta) 1_{\{\rho\geq t\}}\right]+3\eps,
\end{eqnarray}
where $\bar\Rho:=(t,\rho_1(s)\equiv s+1)\in\T_t^2$.
Therefore,
$$\E_{t-\delta}\left[G^1(\Rho[\hat\Tau],\hat\Tau[\Rho],t-\delta)\right]\leq\E_{t-\delta}\left[G^1(\hat\Rho[\hat\Tau],\hat\Tau[\hat\Rho],t-\delta)\right]+11\eps.$$
Similarly, we can show that for any $\Tau\in\T_{t-\delta}^2$,
$$\E_{t-\delta}\left[G^2(\hat\Rho[\Tau],\Tau[\hat\Rho],t-\delta)\right]\leq\E_{t-\delta}\left[G^2(\hat\Rho[\hat\Tau],\hat\Tau[\hat\Rho],t-\delta)\right]+11\eps.$$
\end{proof}

\begin{lemma}\label{l2}
Let $G:[0,\infty]^3\times\Omega\mapsto\R$ such that $G(r,s,t)$ is $\cF_{r\vee s\vee t}$-measurable. Then
$$\inf_{\Rho,\Tau\in\T_t^2}\E_t\left[G(\Rho[\Tau],\Tau[\Rho],t)\right]=\inf_{\rho,\tau\in\cT_t}\E_t\left[G(\rho,\tau,t)\right].$$
\end{lemma}
\begin{proof}
Fix $t\in[0,\infty)$. For any $\Rho,\Tau\in\T_t^2$, $\Rho[\Tau],\Tau[\Rho]\in\cT_t$, and thus
$$\E_t\left[G(\Rho[\Tau],\Tau[\Rho],t)\right]\geq\inf_{\rho,\tau\in\cT_t}\E_t\left[G(\rho,\tau,t)\right].$$
This implies that
$$\inf_{\Rho,\Tau\in\T_t^2}\E_t\left[G(\Rho[\Tau],\Tau[\Rho],t)\right]\geq\inf_{\rho,\tau\in\cT_t}\E_t\left[G(\rho,\tau,t)\right].$$

Conversely, for any $\rho,\tau\in\cT_t$, let
\begin{equation}\label{e2}
\rho_1(s)= 
\begin{cases}
\rho, & \text{if}\ s<\rho,\\
\infty,  & \text{otherwise},
\end{cases}
\quad\text{and}\quad
\tau_1(s)= 
\begin{cases}
\tau, & \text{if}\ s<\tau,\\
\infty,  & \text{otherwise},
\end{cases}
\end{equation}
and $\bar\Rho=(\rho,\rho_1),\bar\Tau=(\tau,\tau_1)$. Then $\bar\Rho,\bar\Tau\in\T_t^2$, and $\bar\Rho[\bar\Tau]=\rho$ and $\bar\Tau[\bar\Rho]=\tau$. Therefore,
$$\inf_{\Rho,\Tau\in\T_t^2}\E_t\left[G(\Rho[\Tau],\Tau[\Rho],t)\right]\leq\E_t\left[G(\rho,\tau,t)\right].$$
This implies that
$$\inf_{\Rho,\Tau\in\T_t^2}\E_t\left[G(\Rho[\Tau],\Tau[\Rho],t)\right]\leq\inf_{\rho,\tau\in\cT_t}\E_t\left[G(\rho,\tau,t)\right].$$
\end{proof}

\begin{lemma}\label{l3}
Let $G:[0,\infty]^3\times\Omega\mapsto\R$ such that $G(r,s,t)$ is $\cF_{r\vee s\vee t}$-measurable and satisfies \asref{a1}. Then for any $\eps>0$, there exists $h>0$, such that for any $t\in[0,\infty)$, there exists $(\hat\Rho,\hat\Tau)\in(\T_t^2)^2$ such that for any $\delta\in[0,h]$, $(\hat\Rho,\hat\Tau)$ is an $\eps$-optimizer for
\begin{equation}\label{e3}
g_{t-\delta}:=\inf_{\Rho,\Tau\in\T_{t-\delta}^2}\E_{t-\delta}\left[G(\Rho[\Tau],\Tau[\Rho],t-\delta)\right].
\end{equation}
That is, for any $\delta\in[0,h]$,
$$\E_{t-\delta}\left[G(\hat\Rho[\hat\Tau],\hat\Tau[\hat\Rho],t-\delta)\right]\leq g_{t-\delta}+\eps.$$
\end{lemma}

\begin{proof}
For any $\eps>0$, let $h>0$ such that $\eta(h)<\eps$. Fix $t\in[0,\infty)$. By \leref{l2},
\begin{equation}\label{e4}
g_t=\inf_{\rho,\tau\in\cT_t}\E_t\left[G(\rho,\tau,t)\right].
\end{equation}
Let $(\rho^*,\tau^*)\in(\cT_t)^2$ be an $\eps$-optimizer for $g_t$ in \eqref{e4}. Define $\hat\Rho,\hat\Tau\in\T_t^2$ as in \eqref{e2} such that $\hat\Rho[\hat\Tau]=\rho^*$ and $\hat\Tau[\hat\Rho]=\tau^*$. Now for any $\delta\in[0,h]$, we have that
\begin{eqnarray}
\notag \E_{t-\delta}\left[U(\hat\Rho[\hat\Tau],\hat\Tau[\hat\Rho],t-\delta)\right]-g_{t-\delta}&=&\E_{t-\delta}\left[U(\rho^*,\tau^*,t-\delta)\right]-g_{t-\delta}\\
\notag &\leq&\E_{t-\delta}\left[\E_t\left[U(\rho^*,\tau^*,t)\right]\right]-g_{t-\delta}+\eps\\
\notag &\leq&\E_{t-\delta}\left[\inf_{\rho,\tau\in\cT_t}\E_t\left[U(\rho,\tau,t)\right]\right]-g_{t-\delta}+2\eps\\
\notag &\leq&\inf_{\rho,\tau\in\cT_t}\E_{t-\delta}\left[\E_t\left[U(\rho,\tau,t)\right]\right]-g_{t-\delta}+2\eps\\
\notag &\leq&\inf_{\rho,\tau\in\cT_t}\E_{t-\delta}\left[U(\rho,\tau,t-\delta)\right]-g_{t-\delta}+3\eps\\
\notag &=&\inf_{\rho,\tau\in\cT_{t-\delta}}\E_{t-\delta}\left[U(\rho\vee t,\tau\vee t,t-\delta)\right]-g_{t-\delta}+3\eps\\
\notag &\leq&\inf_{\rho,\tau\in\cT_{t-\delta}}\E_{t-\delta}\left[U(\rho,\tau,t-\delta)\right]-g_{t-\delta}+5\eps\\
\notag &=& 5\eps.
\end{eqnarray}
This implies that $(\hat\Rho,\hat\Tau)\in(\T_t^2)^2$ is a $5\eps$-optimizer for $g_{t-\delta}$ in \eqref{e3} for any $\delta\in[0,h]$.
\end{proof}

\begin{lemma}\label{l6}
Let $H:[0,\infty]^2\times\Omega\mapsto\R$ such that $H(s,t)$ is $\cF_{s\vee t}$-measurable and satisfies \asref{a1}. Then for any $\eps>0$, there exists $h>0$, such that for any $t\in[0,\infty)$, there exists $\hat\rho\in\cT_t$ such that for any $\delta\in[0,h]$, $\hat\rho$ is an $\eps$-optimizer for
$$\inf_{\rho\in\cT_{t-\delta}}\E_{t-\delta}\left[H(\rho,t-\delta)\right].$$
\end{lemma}

\begin{proof}
The proof is similar to that for \leref{l3} and thus we omit it here.
\end{proof}

\begin{lemma}\label{l4}
Let $G:[0,\infty]^3\times\Omega\mapsto\R$ such that $G(r,s,t)$ is $\cF_{r\vee s\vee t}$-measurable and satisfies \asref{a1}. Then for $t\geq 0$,
\begin{equation}\label{e5}
f_t:=\sup_{\Rho\in\T_t^2}\inf_{\Tau\in\T_t^2}\E_t\left[G(\Rho[\Tau],\Tau[\Rho],t)\right]=\inf_{\Tau\in\T_t^2}\sup_{\Rho\in\T_t^2}\E_t\left[G(\Rho[\Tau],\Tau[\Rho],t)\right].
\end{equation}
Moreover, the process $(f_t)_{t\geq 0}$ is right continuous.
\end{lemma}

\begin{proof}
The equality in \eqref{e5} follows from \cite[Proposition 3.3]{ZZ9}. Now fix $t\in[0,\infty)$ and let $t_n\searrow t$ with $|t_n-t|<1/n$. For any $\eps>0$, by \leref{l1}, there exists $h>0$ and $(\hat\Rho,\hat\Tau)\in(\T_t^2)^2$, such that for any $\delta\in[0,h]$, $(\hat\Rho,\hat\Tau)\in(\T_t^2)^2$ is an $\eps$-saddle point for the game $f_{t-h}$, i.e., for any $\Rho,\Tau\in\T_{t-\delta}^2$,
$$\E_{t-\delta}\left[G(\Rho[\hat\Tau],\hat\Tau[\Rho],t-\delta)\right]-\eps\leq\E_{t-\delta}\left[G(\hat\Rho[\hat\Tau],\hat\Tau[\hat\Rho],t-\delta)\right]\leq\E_{t-\delta}\left[G(\hat\Rho[\Tau],\Tau[\hat\Rho],t-\delta)\right]+\eps.$$
Then for $n$ large enough such that $|t_n-t|<h$, we have that
\begin{eqnarray}
\notag \left|f_{t_n}-f_t\right|&\leq&\left|f_{t_n}-\E_{t_n}\left[G(\hat\Rho[\hat\Tau],\hat\Tau[\hat\Rho],t_n)\right]\right|+\left|\E_{t_n}\left[G(\hat\Rho[\hat\Tau],\hat\Tau[\hat\Rho],t_n)\right]-\E_{t_n}\left[G(\hat\Rho[\hat\Tau],\hat\Tau[\hat\Rho],t)\right]\right|\\
\notag &+&\left|\E_{t_n}\left[G(\hat\Rho[\hat\Tau],\hat\Tau[\hat\Rho],t)\right]-\E_t\left[G(\hat\Rho[\hat\Tau],\hat\Tau[\hat\Rho],t)\right]\right|+\left|\E_t\left[G(\hat\Rho[\Tau],\Tau[\hat\Rho],t)\right]-f_t\right|\\
\notag &\leq&\eps+\eta(1/n)+\left|\E_{t_n}\left[G(\hat\Rho[\hat\Tau],\hat\Tau[\hat\Rho],t)\right]-\E_t\left[G(\hat\Rho[\hat\Tau],\hat\Tau[\hat\Rho],t)\right]\right|+\eps.
\end{eqnarray}
Since the process $(E_{t+s}\left[G(\hat\Rho[\hat\Tau],\hat\Tau[\hat\Rho],t)\right])_{s\geq 0}$ is a martingale w.r.t. the filtration $(\cF_{t+s})_{s\geq 0}$ satisfying the usual conditions, by \cite[Theorem 3.13, page 16]{KS}, it is right continuous. Therefore,
$$\limsup_{n\rightarrow\infty}\left|f_{t_n}-f_t\right|\leq 2\eps.$$
By the arbitrariness of $\eps$, $\lim_{n\rightarrow\infty}f_{t_n}=f_t$, and thus the conclusion follows.
\end{proof}

\begin{lemma}\label{l5}
Let $G:[0,\infty]^3\times\Omega\mapsto\R$ such that $G(r,s,t)$ is $\cF_{r\vee s\vee t}$-measurable and satisfies \asref{a1}. Then the process $(g_t)_{t\geq 0}$ defined by
$$g_t:=\inf_{\Rho,\Tau\in\T_t^2}\E_t\left[G(\Rho[\Tau],\Tau[\Rho],t)\right]$$
is right continuous.
\end{lemma}

\begin{proof}
The proof is similar to that for \leref{l4} and thus we omit it here.
\end{proof}

\begin{proposition}\label{p1}
Let $U:[0,\infty]^3\times\Omega\mapsto\R$ such that $U(r,s,t)$ is $\cF_{r\vee s\vee t}$-measurable and satisfies \asref{a1}. For $t\geq 0$, let
\begin{equation}\label{e6}
\underline V_t:=\sup_{\Rho\in\T_t^3}\inf_{\Tau,\Sig\in\T_t^3}\E_t\left[U(\Rho[\Tau,\Sig],\Tau[\Rho,\Sig],\Sig[\Rho,\Tau])\right]
\end{equation}
and
$$\overline V_t:=\inf_{\Tau,\Sig\in\T_t^3}\sup_{\Rho\in\T_t^3}\E_t\left[U(\Rho[\Tau,\Sig],\Tau[\Rho,\Sig],\Sig[\Rho,\Tau])\right].$$
Then for any $\mu\in\cT$,
\begin{equation}\label{e7}
\underline V_\mu=\overline V_\mu=\sup_{\rho\in\cT_\mu}\inf_{\theta\in\cT_\mu}\E_\mu\left[X_\rho 1_{\{\rho\leq\theta\}}+Y_\theta 1_{\{\rho>\theta\}}\right]=\inf_{\theta\in\cT_\mu}\sup_{\rho\in\cT_\mu}\E_\mu\left[X_\rho 1_{\{\rho\leq\theta\}}+Y_\theta 1_{\{\rho>\theta\}}\right],
\end{equation}
where for $t\geq 0$,
$$X_t:=\inf_{\Tau,\Sig\in\T_t^2}\E_t\left[U(t,\Tau[\Sig],\Sig[\Tau])\right],$$
and
$$Y_t:=\left(\sup_{\Rho\in\T_t^2}\inf_{\Sig\in\T_t^2}\E_t\left[U(\Rho[\Sig],t,\Sig[\Rho])\right]\right)\wedge\left(\sup_{\Rho\in\T_t^2}\inf_{\Tau\in\T_t^2}\E_t\left[U(\Rho[\Tau],\Tau[\Rho],t)\right]\right).$$
Moreover, there exists an $\eps$-saddle point for \eqref{e6} with $t$ replaced by $\mu$ for any $\eps>0$.
\end{proposition}

\begin{proof}
Denote
$$Y_t^2:=\sup_{\Rho\in\T_t^2}\inf_{\Sig\in\T_t^2}\E_t\left[U(\Rho[\Sig],t,\Sig[\Rho])\right]\quad\text{and}\quad Y_t^3:=\sup_{\Rho\in\T_t^2}\inf_{\Tau\in\T_t^2}\E_t\left[U(\Rho[\Tau],\Tau[\Rho],t)\right].$$
For $t\geq 0$, let $\pmb\theta=(t,\theta_1(s)\equiv s+1)\in\T_t^2$. Then
$$X_t\leq\inf_{\Sig\in\T_t^2}\E_t\left[U(t,t,\Sig[\pmb\theta])\right]\leq Y_t^2.$$
Similarly, we can show that $X_t\leq Y_t^3$. Hence,
$$X_t\leq Y_t.$$
Moreover, by Lemmas \ref{l4} and \ref{l5}, $(X_t)_{t\geq 0}$ and $(Y_t)_{t\geq 0}$ are right continuous. Then from the classical theory of Dynkin games, we have that for $t\geq 0$,
\begin{eqnarray}
\label{e8} V_t&:=&\sup_{\rho\in\cT_t}\inf_{\theta\in\cT_t}\E_t\left[X_\rho 1_{\{\rho\leq\theta\}}+Y_\theta 1_{\{\rho>\theta\}}\right]=\inf_{\theta\in\cT_t}\sup_{\rho\in\cT_t}\E_t\left[X_\rho 1_{\{\rho\leq\theta\}}+Y_\theta 1_{\{\rho>\theta\}}\right]\\
\notag &=&\sup_{\rho\in\cT_t}\inf_{\theta\in\cT_t}\E_t\left[X_\rho 1_{\{\rho<\theta\}}+Y_\theta 1_{\{\rho\geq\theta\}}\right]=\inf_{\theta\in\cT_t}\sup_{\rho\in\cT_t}\E_t\left[X_\rho 1_{\{\rho<\theta\}}+Y_\theta 1_{\{\rho\geq\theta\}}\right],
\end{eqnarray}
and the process $(V_t)_{t\geq 0}$ is right continuous. 

Now fix $\mu\in\cT$ and let $\eps>0$. Define $\hat\rho,\hat\theta\in\cT_\mu$ as
\begin{equation}\label{e9}
\hat\rho:=\inf\{t\geq \mu:\ V_t\leq X_t+\eps\}\quad\text{and}\quad\hat\theta:=\inf\{t\geq \mu:\ V_t\geq Y_t-\eps\}.
\end{equation}
Then $(\hat\rho,\hat\theta)\in(\cT_\mu)^2$ is an $\eps$-saddle point for the Dynkin game $V_\mu$ defined in \eqref{e8} with $t$ replaced by $\mu$.  By \leref{l3}, there exists $h^x>0$ such that for any $t\geq 0$, there exists $(\hat\Tau_t^1=(\hat\tau_t^1,\hat\tau_{t,3}^1),\hat\Sig_t^1=(\hat\sigma_t^1,\hat\sigma_{t,2}^1))\in(\T_t^2)^2$ being an $\eps$-optimizer for $X_{t'}$ for any $t'\in[t-h^x,t]$. Similarly, by \leref{l1}, there exists $h^y>0$, such that for $t\geq 0$, there exist $(\hat\Rho_t^2=(\hat\rho_t^2,\hat\rho_{t,3}^2),\hat\Sig_t^2=(\hat\sigma_t^2,\sigma_{t,1}^2)),(\hat\Rho_t^3=(\hat\rho_t^3,\hat\rho_{t,2}^3),\hat\Tau_t^3=(\hat\tau_t^3,\hat\tau_{t,1}^3))\in(\T_t^2)^2$ being $\eps$-saddle points for $Y_{t'}^2$ and $Y_{t'}^3$ respectively for any $t'\in[t-h^y,t]$. Furthermore, by \leref{l6} there exists $h^z>0$ such that for any $t\geq 0$, there exist $\hat\rho_t^{23},\hat\tau_t^{13},\hat\sigma_t^{12}\in\cT_t$ being $\eps$-optimizers for
$$Z_{t'}^{23}:=\sup_{\rho\in\cT_{t'}}\E_{t'}\left[U(\rho,t',t')\right],\quad Z_{t'}^{13}:=\inf_{\tau\in\cT_{t'}}\E_{t'}\left[U(t',\tau,t')\right],\quad Z_{t'}^{12}:=\inf_{\sigma\in\cT_{t'}}\E_{t'}\left[U(t',t',\sigma)\right]$$
respectively for any $t'\in[t-h^z,t]$. Let $h:=h^x\wedge h^y\wedge h^z$, and
$$A:=\{V_{\hat\theta}\geq Y_{\hat\theta}^2-\eps\}.$$
For any $t\geq 0$, define
\begin{equation}\label{e17}
\phi_h(t):=([t/h]+1)h
\end{equation}
Observe that $\phi_t(\cdot)$ is right continuous and $\phi_h(t)>t$ for $t<\infty$.

Recall $\hat\rho$ defined in \eqref{e9}. Now define $\hat\Rho=(\hat\rho,\hat\rho_2,\hat\rho_3,\hat\rho_{23}), \hat\Tau=(\hat\tau,\hat\tau_1,\hat\tau_3,\hat\tau_{13}), \hat\Sig=(\hat\sigma,\hat\sigma_1,\hat\sigma_2,\hat\sigma_{12})$ as follows.
$$\hat\tau:=
\begin{cases}
\hat\theta, & \text{on}\ A,\\
\infty, & \text{on}\ A^c,
\end{cases}
\quad\text{and}\quad
\hat\sigma:=
\begin{cases}
\hat\theta, & \text{on}\ A^c,\\
\infty, & \text{on}\ A,
\end{cases}
$$
and
$$\hat\rho_2(t):=\hat\rho_{\phi_h(t)}^2,\quad\hat\rho_3(t):=\hat\rho_{\phi_h(t)}^3,\quad\hat\rho_{23}(s,t):=
\begin{cases}
\hat\rho_{\phi_h(s),3}^2(t), & \text{if}\ s<t,\\
\hat\rho_{\phi_h(t),2}^3(s), & \text{if}\ s>t,\\
\hat\rho^{23}_{\phi_h(s)}, & \text{if}\ s=t,
\end{cases}
$$
and
$$
\hat\tau_1(t):=\hat\tau_{\phi_h(t)}^1,\quad\hat\tau_3(t):=\hat\tau_{\phi_h(t)}^3,\quad\hat\tau_{13}(s,t):=
\begin{cases}
\hat\tau_{\phi_h(s),3}^1(t), & \text{if}\ s<t,\\
\hat\tau_{\phi_h(t),1}^3(s), & \text{if}\ s>t,\\
\hat\tau^{13}_{\phi_h(s)}, & \text{if}\ s=t,
\end{cases}
$$
and
$$
\hat\sigma_1(t):=\hat\sigma_{\phi_h(t)}^1,\quad\hat\sigma_2(t):=\hat\sigma_{\phi_h(t)}^2,\quad\hat\sigma_{12}(s,t):=
\begin{cases}
\hat\sigma_{\phi_h(s),2}^1(t), & \text{if}\ s<t,\\
\hat\sigma_{\phi_h(t),1}^2(s), & \text{if}\ s>t,\\
\hat\sigma^{12}_{\phi_h(s)}, & \text{if}\ s=t.
\end{cases}
$$
It can be shown that $\hat\Rho,\hat\Tau,\hat\Sig\in\T_\mu^3$. In the rest of the proof, we will show that $(\hat\Rho,(\hat\Tau,\hat\Sig))$ is a $17\eps$-saddle point for \eqref{e6} with $t$ replaced by $\mu$.

\textbf{Part 1}. We show that
\begin{equation}\label{e12}
\left|\E_\mu\left[U(\hat\Rho[\hat\Tau,\hat\Sig],\hat\Tau[\hat\Rho,\hat\Sig],\hat\Sig[\hat\Rho,\hat\Tau])\right]-V_\mu\right|\leq 8\eps,
\end{equation}
where $V$ is defined in \eqref{e8}. We consider five cases. 

\textbf{Case 1.1}: $A_1:=\{\hat\rho<\hat\theta\}$.
\begin{eqnarray}
\notag \E_\mu\left[U(\hat\Rho[\hat\Tau,\hat\Sig],\hat\Tau[\hat\Rho,\hat\Sig],\hat\Sig[\hat\Rho,\hat\Tau])1_{A_1}\right]&=&\E_\mu\left[\E_{\hat\rho}\left[U\left(\hat\rho,\hat\Tau^1_{\phi_h(\hat\rho)}\left[\hat\Sig_{\phi_h(\hat\rho)}^1\right],\hat\Sig_{\phi_h(\hat\rho)}^1\left[\hat\Tau_{\phi_h(\hat\rho)}^1\right]\right)\right]1_{A_1}\right]\\
\notag &\in&\left[\E_\mu\left[X_{\hat\rho}1_{A_1}\right],\E_t\left[X_{\hat\rho}1_{A_1}\right]+\eps\right].
\end{eqnarray}

\textbf{Case 1.2}: $A_2:=\{\hat\rho=\hat\theta\}\cap A$.
It can be shown that for any $t\geq 0$,
\begin{equation}\label{e14}
X_t\leq Z_t^{12}\leq Y_t^2.
\end{equation}
Hence, on $A$,
\begin{equation}\label{e10}
Z_{\hat\theta}^{12}\leq Y_{\hat\theta}^2\leq V_{\hat\theta}+\eps\leq Y_{\hat\theta}+\eps.
\end{equation}
Then
$$\E_\mu\left[U(\hat\Rho[\hat\Tau,\hat\Sig],\hat\Tau[\hat\Rho,\hat\Sig],\hat\Sig[\hat\Rho,\hat\Tau])1_{A_2}\right]=\E_\mu\left[\E_{\hat\theta}\left[U\left(\hat\theta,\hat\theta,\hat\Sig_{\phi_h(\hat\theta)}^{12}\right)\right]1_{A_2}\right]\in\left[\E_\mu\left[X_{\hat\theta} 1_{A_2}\right],\E_\mu\left[Y_{\hat\theta} 1_{A_2}\right]+2\eps\right].$$

\textbf{Case 1.3}: $A_3:=\{\hat\rho=\hat\theta\}\cap A^c$.
It can be shown that for any $t\geq 0$,
\begin{equation}\label{e15}
X_t\leq Z_t^{13}\leq Y_t^3.
\end{equation}
By the definition of $\hat\theta$ in \eqref{e9}, we have that on $A^c$,
$$Y_{\hat\theta}^2\wedge Y_{\hat\theta}^3-\eps\leq V_{\hat\theta}<Y_{\hat\theta}^2-\eps.$$
This implies that on $A^c$,
\begin{equation}\label{e11}
Y_{\hat\theta}^2>Y_{\hat\theta}^3=Y_{\hat\theta}^2\wedge Y_{\hat\theta}^3=Y_{\hat\theta}.
\end{equation}
Then
$$\E_\mu\left[U(\hat\Rho[\hat\Tau,\hat\Sig],\hat\Tau[\hat\Rho,\hat\Sig],\hat\Sig[\hat\Rho,\hat\Tau])1_{A_3}\right]=\E_\mu\left[\E_{\hat\theta}\left[U\left(\hat\theta,\hat\Tau_{\phi_h(\hat\theta)}^{13},\hat\theta\right)\right]1_{A_3}\right]\in\left[\E_\mu\left[X_{\hat\theta} 1_{A_3}\right],\E_\mu\left[Y_{\hat\theta} 1_{A_3}\right]+\eps\right].$$

\textbf{Case 1.4}: $A_4:=\{\hat\rho>\hat\theta\}\cap A$. By \eqref{e10}, we have that
\begin{eqnarray}
\notag \E_\mu\left[U(\hat\Rho[\hat\Tau,\hat\Sig],\hat\Tau[\hat\Rho,\hat\Sig],\hat\Sig[\hat\Rho,\hat\Tau])1_{A_4}\right]&=&\E_\mu\left[\E_{\hat\theta}\left[U\left(\hat\Rho^2_{\phi_h(\hat\theta)}\left[\hat\Sig_{\phi_h(\hat\theta)}^2\right],\hat\theta,\hat\Sig_{\phi_h(\hat\theta)}^2\left[\hat\Rho_{\phi_h(\hat\theta)}^2\right]\right)\right]1_{A_4}\right]\\
\notag &\in&\left[\E_\mu\left[Y_{\hat\theta}^2 1_{A_4}\right]-\eps,\E_\mu\left[Y_{\hat\theta}^2 1_{A_4}\right]+\eps\right]\\
\notag &\subset&\left[\E_\mu\left[Y_{\hat\theta} 1_{A_4}\right]-\eps,\E_\mu\left[Y_{\hat\theta} 1_{A_4}\right]+2\eps\right].
\end{eqnarray}

\textbf{Case 1.5}: $A_5:=\{\hat\rho>\hat\theta\}\cap A^c$. By \eqref{e11}, we have that
\begin{eqnarray}
\notag \E_\mu\left[U(\hat\Rho[\hat\Tau,\hat\Sig],\hat\Tau[\hat\Rho,\hat\Sig],\hat\Sig[\hat\Rho,\hat\Tau])1_{A_5}\right]&=&\E_\mu\left[\E_{\hat\theta}\left[U\left(\hat\Rho^3_{\phi_h(\hat\theta)}\left[\hat\Tau_{\phi_h(\hat\theta)}^3\right],\hat\Tau_{\phi_h(\hat\theta)}^3\left[\hat\Rho_{\phi_h(\hat\theta)}^3\right],\hat\theta\right)\right]1_{A_5}\right]\\
\notag &\in&\left[\E_\mu\left[Y_{\hat\rho}^3 1_{A_5}\right]-\eps,\E_\mu\left[Y_{\hat\rho}^3 1_{A_5}\right]+\eps\right]\\
\notag &\subset&\left[\E_\mu\left[Y_{\hat\rho} 1_{A_4}\right]-\eps,\E_\mu\left[Y_{\hat\rho} 1_{A_5}\right]+\eps\right].
\end{eqnarray}
From cases 1.1-1.5, we have that
\begin{eqnarray}
\notag V_\mu-3\eps&\leq&\E_\mu\left[X_{\hat\rho} 1_{\{\hat\rho\leq\hat\theta\}}+Y_{\hat\theta} 1_{\{\hat\rho>\hat\theta\}}\right]-2\eps\\
\notag &\leq&\E_\mu\left[U(\hat\Rho[\hat\Tau,\hat\Sig],\hat\Tau[\hat\Rho,\hat\Sig],\hat\Sig[\hat\Rho,\hat\Tau])\right]\\
\notag &\leq&\E_\mu\left[X_{\hat\rho} 1_{\{\hat\rho<\hat\theta\}}+Y_{\hat\theta} 1_{\{\hat\rho\geq\hat\theta\}}\right]+7\eps\leq V_\mu+8\eps.
\end{eqnarray}

\textbf{Part 2}: We show that for any $\Rho\in\T_t^3$,
\begin{equation}\label{e13}
\E_\mu\left[U(\Rho[\hat\Tau,\hat\Sig],\hat\Tau[\Rho,\hat\Sig],\hat\Sig[\Rho,\hat\Tau])\right]\leq V_\mu+9\eps.
\end{equation}
Take $\Rho=(\rho,\rho_2,\rho_3,\rho_{23})\in\T_\mu^3$. We consider five cases.

\textbf{Case 2.1}: $B_1:=\{\rho<\hat\theta\}$.
\begin{eqnarray}
\notag \E_\mu\left[U(\Rho[\hat\Tau,\hat\Sig],\hat\Tau[\Rho,\hat\Sig],\hat\Sig[\Rho,\hat\Tau])1_{B_1}\right]&=&\E_\mu\left[\E_{\rho}\left[U\left(\rho,\hat\Tau^1_{\phi_h(\rho)}\left[\hat\Sig_{\phi_h(\rho)}^1\right],\hat\Sig_{\phi_h(\rho)}^1\left[\hat\Tau_{\phi_h(\rho)}^1\right]\right)\right]1_{B_1}\right]\\
\notag &\leq&\E_\mu\left[X_\rho 1_{B_1}\right]+\eps.
\end{eqnarray}

\textbf{Case 2.2}: $B_2:=\{\rho=\hat\theta\}\cap A$. By \eqref{e14} and \eqref{e10},
$$\E_\mu\left[U(\Rho[\hat\Tau,\hat\Sig],\hat\Tau[\Rho,\hat\Sig],\hat\Sig[\Rho,\hat\Tau])1_{B_2}\right]=\E_\mu\left[\E_{\hat\theta}\left[U\left(\hat\theta,\hat\theta,\hat\Sig_{\phi_h(\hat\theta)}^{12}\right)\right]1_{A_2}\right]\leq\E_\mu\left[Y_{\hat\theta} 1_{B_2}\right]+2\eps.$$

\textbf{Case 2.3}: $B_3:=\{\rho=\hat\theta\}\cap A^c$. By \eqref{e15} and \eqref{e11},
$$\E_\mu\left[U(\Rho[\hat\Tau,\hat\Sig],\hat\Tau[\Rho,\hat\Sig],\hat\Sig[\Rho,\hat\Tau])1_{B_3}\right]=\E_\mu\left[\E_{\hat\theta}\left[U\left(\hat\theta,\hat\Tau_{\phi_h(\hat\theta)}^{13},\hat\theta\right)\right]1_{B_3}\right]\leq\E_\mu\left[Y_{\hat\theta} 1_{B_3}\right]+\eps.$$

\textbf{Case 2.4}: $B_4:=\{\rho>\hat\theta\}\cap A$. Define $\Rho_{\hat\theta}^2:=(\rho_2(\hat\theta),\rho_{23}(\hat\theta,\cdot))\in\T_{\hat\theta}^2$. We have that
\begin{eqnarray}
\notag \E_\mu\left[U(\Rho[\hat\Tau,\hat\Sig],\hat\Tau[\Rho,\hat\Sig],\hat\Sig[\Rho,\hat\Tau])1_{B_4}\right]&=&\E_\mu\left[\E_{\hat\theta}\left[U\left(\Rho_{\hat\theta}^2\left[\hat\Sig_{\phi_h(\hat\theta)}^2\right],\hat\theta,\hat\Sig_{\phi_h(\hat\theta)}^2\left[\Rho_{\hat\theta}^2\right]\right)\right]1_{B_4}\right]\\
\notag &\leq&\E_\mu\left[Y_{\hat\theta}^2 1_{B_4}\right]+\eps\\
\notag &\leq&\E_\mu\left[Y_{\hat\theta} 1_{B_4}\right]+2\eps.
\end{eqnarray}

\textbf{Case 2.5}: $B_5:=\{\rho>\hat\theta\}\cap A^c$. Define $\Rho_{\hat\theta}^3:=(\rho_3(\hat\theta),\rho_{23}(\cdot,\hat\theta))\in\T_{\hat\theta}^2$. We have that
\begin{eqnarray}
\notag \E_\mu\left[U(\Rho[\hat\Tau,\hat\Sig],\hat\Tau[\Rho,\hat\Sig],\hat\Sig[\Rho,\hat\Tau])1_{B_5}\right]&=&\E_\mu\left[\E_{\hat\theta}\left[U\left(\Rho_{\hat\theta}^3\left[\hat\Tau_{\phi_h(\hat\theta)}^3\right],\hat\Tau_{\phi_h(\hat\theta)}^3\left[\Rho_{\hat\theta}^3\right],\hat\theta\right)\right]1_{B_5}\right]\\
\notag &\leq&\E_\mu\left[Y_{\hat\theta}^3 1_{B_5}\right]+\eps\\
\notag &=&\E_\mu\left[Y_{\hat\theta} 1_{B_5}\right]+\eps.
\end{eqnarray}

From cases 2.1-2.5, we have that
\begin{eqnarray}
\notag \E_\mu\left[U(\Rho[\hat\Tau,\hat\Sig],\hat\Tau[\Rho,\hat\Sig],\hat\Sig[\Rho,\hat\Tau])\right]&\leq&\E_\mu\left[X_\rho 1_{\{\rho<\hat\theta\}}+Y_{\hat\theta} 1_{\{\rho\geq\hat\theta\}}\right]+7\eps\\
\notag &\leq& \E_\mu\left[X_{\hat\rho} 1_{\{\hat\rho<\hat\theta\}}+Y_{\hat\theta} 1_{\{\hat\rho\geq\hat\theta\}}\right]+8\eps\\
\notag &\leq& V_\mu+9\eps.
\end{eqnarray}

\textbf{Part 3}: We show that for any $(\Tau,\Sig)\in(\T_\mu^3)^2$,
\begin{equation}\label{e16}
\E_\mu\left[U(\hat\Rho[\Tau,\Sig],\Tau[\hat\Rho,\Sig],\Sig[\hat\Rho,\Tau])\right]\geq V_\mu-5\eps.
\end{equation}
Take $(\Tau=(\tau,\tau_1,\tau_3,\tau_{13}),\Sig=(\sigma,\sigma_1,\sigma_2,\sigma_{12}))\in(\T_\mu^3)^2$. We consider four cases.

\textbf{Case 3.1}: $C_1:=\{\hat\rho\leq\tau\wedge\sigma\}$.
$$\E_\mu\left[U(\hat\Rho[\Tau,\Sig],\Tau[\hat\Rho,\Sig],\Sig[\hat\Rho,\Tau])1_{C_1}\right]=\E_\mu\left[\E_{\hat\rho}\left[U(\hat\rho,\Tau[\hat\Rho,\Sig],\Sig[\hat\Rho,\Tau])\right]1_{C_1}\right]\geq\E_\mu\left[X_{\hat\rho}1_{C_1}\right].$$

\textbf{Case 3.2}: $C_2:=\{\hat\rho>\tau\wedge\sigma\}\cap\{\tau=\sigma\}$. It can be shown that for any $t\geq 0$,
$$Z_t^{23}\geq Y_t^2,Y_t^3\geq Y_t.$$
Then we have that
\begin{eqnarray}
\notag \E_\mu\left[U(\hat\Rho[\Tau,\Sig],\Tau[\hat\Rho,\Sig],\Sig[\hat\Rho,\Tau])1_{C_2}\right]&=&\E_\mu\left[\E_\tau\left[U\left(\hat\rho_{\phi_h(\tau)}^{23},\tau,\tau\right)\right]1_{C_2}\right]\\
\notag &\geq&\E_\mu\left[Z_\tau^{23}1_{C_2}\right]-\eps\\
\notag &\geq&\E_\mu\left[Y_{\tau\wedge\sigma} 1_{C_2}\right]-\eps.
\end{eqnarray}

\textbf{Case 3.3}: $C_3:=\{\hat\rho>\tau\wedge\sigma\}\cap\{\tau<\sigma\}$. Define $\Sig_\tau^2:=(\sigma_2(\tau),\sigma_{12}(\cdot,\tau)\in\T_\tau^2$. Then
\begin{eqnarray}
\notag \E_\mu\left[U(\hat\Rho[\Tau,\Sig],\Tau[\hat\Rho,\Sig],\Sig[\hat\Rho,\Tau])1_{C_3}\right]&=&\E_\mu\left[\E_\tau\left[U\left(\hat\Rho^2_{\phi_h(\tau)}\left[\Sig_\tau^2\right],\tau,\Sig_\tau^2\left[\hat\Rho_{\phi_h(\tau)}^2\right]\right)\right]1_{C_3}\right]\\
\notag &\geq&\E_\mu\left[Y_\tau^2 1_{C_3}\right]-\eps\\
\notag &\geq&\E_\mu\left[Y_{\tau\wedge\sigma} 1_{C_3}\right]-\eps.
\end{eqnarray}

\textbf{Case 3.4}: $C_4:=\{\hat\rho>\tau\wedge\sigma\}\cap\{\tau>\sigma\}$. Similar to case 3.3, we can show that
$$\E_\mu\left[U(\hat\Rho[\Tau,\Sig],\Tau[\hat\Rho,\Sig],\Sig[\hat\Rho,\Tau])1_{C_4}\right]\leq\E_\mu\left[Y_{\tau\wedge\sigma} 1_{C_4}\right]-\eps.$$
From cases 3.1-3.4, we have that
\begin{eqnarray}
\notag \E_\mu\left[U(\hat\Rho[\Tau,\Sig],\Tau[\hat\Rho,\Sig],\Sig[\hat\Rho,\Tau])\right]&\geq&\E_\mu\left[X_{\hat\rho} 1_{\{\hat\rho\leq\tau\wedge\sigma\}}+Y_{\tau\wedge\sigma} 1_{\{\hat\rho>\tau\wedge\sigma\}}\right]-3\eps\\
\notag &\geq&\E_\mu\left[X_{\hat\rho} 1_{\{\hat\rho\leq\hat\theta\}}+Y_{\tau\wedge\sigma} 1_{\{\hat\rho>\hat\theta\}}\right]-4\eps\\
\notag &\geq&V_\mu-5\eps.
\end{eqnarray}
By \eqref{e12},\eqref{e15} and \eqref{e16}, $(\hat\Rho,(\hat\Tau,\hat\Sig))\in(\T_\mu^3)^3$ is a $17\eps$-saddle point for \eqref{e6} with $t$ replaced by $\mu$.
\end{proof}

\section{Proof of \thref{t1}}

In this section, we will construct an $\eps$-Nash equilibrium of the three-player game by using $\eps$-Nash equilibriums of two-player games, as well as $\eps$-saddle points of games like \eqref{e6}.

Let $\eps>0$ be fixed. By lemmas \ref{l2}-\ref{l6}, we can choose $h>0$, such that for any $t$, there exist $(\overline\Tau_t^1=(\overline\tau_t^1,\overline\tau_{t,3}^1),\overline\Sig_t^1=(\overline\sigma_t^1,\overline\sigma_{t,2}^1))\in(\T_t)^2$ being an $\eps$-Nash equilibrium for the game
$$\E_{t'}\left[U^i(t',\Tau[\Sig],\Sig[\Tau])\right],\quad\Tau,\Sig\in\T_{t'},\quad i=2,3,$$
for any $t'\in[t-h,t]$, and $(\overline\Rho_t^2=(\overline\rho_t^2,\overline\rho_{t,3}^2),\overline\Sig_t^2=(\overline\sigma_t^2,\overline\sigma_{t,1}^2))\in(\T_t)^2$ being an $\eps$-Nash equilibrium for the game
$$\E_{t'}\left[U^i(\Rho[\Sig],t',\Sig[\Rho])\right],\quad\Rho,\Sig\in\T_{t'},\quad i=1,3,$$
for any $t'\in[t-h,t]$, and $(\overline\Rho_t^3=(\overline\rho_t^3,\overline\rho_{t,2}^3),\overline\Tau_t^3=(\overline\tau_t^3,\overline\tau_{t,1}^3))\in(\T_t)^2$ being an $\eps$-Nash equilibrium for the game
$$\E_{t'}\left[U^i(\Rho[\Tau],\Tau[\Rho],t')\right],\quad\Rho,\Tau\in\T_{t'},\quad i=1,2,$$
for any $t'\in[t-h,t]$, and $\underline\rho_t^{i,23}\in\cT_t$ being an $\eps$-optimizer for
$$\inf_{\rho\in\cT_{t'}}\E_{t'}\left[U^i(\rho,t',t')\right]$$
for any $t'\in[t-h,t]$ for $i=2,3$, and $\underline\tau_t^{i,13}\in\cT_t$ being an $\eps$-optimizer for 
$$\inf_{\tau\in\cT_{t'}}\E_{t'}\left[U^i(t',\tau,t')\right]$$
for any $t'\in[t-h,t]$ for $i=1,3$, and $\underline\sigma_t^{i,12}\in\cT_t$ being an $\eps$-optimizers for
$$\inf_{\sigma\in\cT_{t'}}\E_{t'}\left[U^i(t',t',\sigma)\right].$$

Recall $\phi_h(\cdot)$ defined in \eqref{e17}. For $t\geq 0$, define
\begin{eqnarray}
\notag X_t^1&:=&\inf_{\Tau,\Sig\in\T_t}\E_t\left[U^1(t,\Tau[\Sig],\Sig[\Tau])\right],\\
\notag Z_t^1&:=&\E_t\left[U^1\left(t,\overline\Tau_{\phi_h(t)}^1\left[\overline\Sig_{\phi_h(t)}^1\right],\overline\Sig_{\phi_h(t)}^1\left[\overline\Tau_{\phi_h(t)}^1 \right]\right)\right],\\
\notag Y_t^{1,2}&:=&\E_t\left[U^1\left(\overline\Rho_{\phi_h(t)}^2\left[\overline\Sig_{\phi_h(t)}^2\right],t,\overline\Sig_{\phi_h(t)}^2\left[\overline\Rho_{\phi_h(t)}^2\right]\right)\right],\\
\notag Y_t^{1,3}&:=&\E_t\left[U^1\left(\overline\Rho_{\phi_h(t)}^3\left[\overline\Tau_{\phi_h(t)}^3\right],\overline\Tau_{\phi_h(t)}^3\left[\overline\Rho_{\phi_h(t)}^3\right],t\right)\right],\\
\notag Y_t^1&:=&Y_t^{1,2}\wedge Y_t^{1,3}+\eps,\\
\notag V_t^1&:=&\sup_{\rho\in\cT_\theta}\inf_{\lambda\in\cT_\theta}\E_t\left[X_\rho^1 1_{\{\rho\leq\lambda\}}+Y_\lambda^1 1_{\{\rho>\lambda\}}\right],\\
\notag \mu^1&:=&\inf\{t\geq\theta:\ V_t^1\leq W_t^1+\eps\},
\end{eqnarray}
and $X_t^2,Z_t^2,Y_t^{2,1},Y_t^{2,3},Y_t^2,V_t^2,\mu^2,X_t^3,Z_t^3,Y_t^{3,1},Y_t^{3,2},Y_t^3,V_t^3,\mu^3$ in a symmetric way.

\begin{lemma}\label{l8}
For $i=1,2,3$ and any $t\geq 0$, $X_t^i\leq Y_t^i$.
\end{lemma}

\begin{proof}
For $t\geq 0$, let $\tilde\Rho=\Tilde\Tau=(t,\infty)\in\T_t^2$. Then
\begin{eqnarray}
\notag X_t^1&\leq&\E_t\left[U^1\left(t,\tilde\Tau\left[\overline\Sig_{\phi_h(t)}^2\right],\overline\Sig_{\phi_h(t)}^2\left[\tilde\Tau\right]\right)\right]\\
\notag &=&\E_t\left[U^1\left(t,t,\overline\Sig_{\phi_h(t)}^2\left[\tilde\Tau\right]\right)\right]\\
\notag &=&\E_t\left[U^1\left(\tilde\Rho\left[\overline\Sig_{\phi_h(t)}^2\right],t,\overline\Sig_{\phi_h(t)}^2\left[\tilde\Rho\right]\right)\right]\leq Y_t^{1,2}+\eps.
\end{eqnarray}
Similarly, we can show that $X_t^1\leq Y_t^{1,3}+\eps$. Hence, $X_t^1\leq Y_t^1$.
\end{proof}

\begin{lemma}\label{l9}
For $i=1,2,3$, the processes $(Z_t^i)_{t\geq 0}$ and $(Y_t^i)_{t\geq 0}$ are right continuous.
\end{lemma}
\begin{proof}
Let $t\in[0,\infty)$. Let $t^n\searrow t$. Without loss of generality, we assume that $|t_n-t|<(1/n)\wedge(\phi_h(t)-t)$. Then
\begin{eqnarray}
\notag &&\hspace{-0.7cm} |Z_{t_n}^1-Z_t^1|\\
\notag &&\hspace{-0.7cm} =\left|\E_{t_n}\left[U^1\left(t_n,\overline\Tau_{\phi_h(t)}^1\left[\overline\Sig_{\phi_h(t)}^1\right],\overline\Sig_{\phi_h(t)}^1\left[\overline\Tau_{\phi_h(t)}^1 \right]\right)\right]-\E_t\left[U^1\left(t,\overline\Tau_{\phi_h(t)}^1\left[\overline\Sig_{\phi_h(t)}^1\right],\overline\Sig_{\phi_h(t)}^1\left[\overline\Tau_{\phi_h(t)}^1 \right]\right)\right]\right|\\
\notag &&\hspace{-0.7cm} \leq\left|\E_{t_n}\left[U^1\left(t,\overline\Tau_{\phi_h(t)}^1\left[\overline\Sig_{\phi_h(t)}^1\right],\overline\Sig_{\phi_h(t)}^1\left[\overline\Tau_{\phi_h(t)}^1 \right]\right)\right]-\E_t\left[U^1\left(t,\overline\Tau_{\phi_h(t)}^1\left[\overline\Sig_{\phi_h(t)}^1\right],\overline\Sig_{\phi_h(t)}^1\left[\overline\Tau_{\phi_h(t)}^1 \right]\right)\right]\right|+\eta(1/n).
\end{eqnarray}
Since
$$\left(\E_{t+s}\left[U^1\left(t,\overline\Tau_{\phi_h(t)}^1\left[\overline\Sig_{\phi_h(t)}^1\right],\overline\Sig_{\phi_h(t)}^1\left[\overline\Tau_{\phi_h(t)}^1 \right]\right)\right]\right)_{s\geq 0}$$
is a martingale w.r.t. the filtration $(\cF_{t+s})_{s\geq 0}$ satisfying the usual conditions, it is right continuous. Hence, $Z_{t_n}^1\rightarrow Z_t^1$. Similarly, we can show that $(Y_t^{1,2})_{t\geq 0}$ and $(Y_t^{1,3})_{t\geq 0}$ are right continuous.
\end{proof}

By Lemmas \ref{l5} and \ref{l9}, we have that for $i=1,2,3$, the process $(V_t^i)_{t\geq 0}$ is right continuous. Then we can choose $\delta>0$ being $\cF_\theta$-measurable, such that for $i=1,2,3$,
\begin{equation}
\label{e18}
\E_\theta\left[\sup_{0\leq r\leq\delta}\left|Z_{\mu^i+r}^i-Z_{\mu^i}^i\right|\right]<\eps\quad\text{and}\quad \E_\theta\left|V_{\mu^i+\delta}^i-V_{\mu^i}^i\right|<\eps.
\end{equation}
It can be shown that for any $\lambda\in\cT_\theta$, $\lambda+\delta\in\cT_{\theta+}$. 

By \prref{p1}, there exist $(\Rho^1, (\underline\Tau^1=(\underline\tau^1,\underline\tau_1^1,\underline\tau_3^1,\underline\tau_{13}^1),\underline\Sig^1=(\underline\sigma^1,\underline\sigma_1^1,\underline\sigma_2^1,\underline\sigma_{12}^1)))\in(\T_{\mu^1+\delta}^3)^3$ being an $\eps$-saddle point for the game
\begin{equation}\label{e19}
\tilde V_{\mu^1+\delta}:=\sup_{\Rho\in\T_{\mu^1+\delta}^3}\inf_{\Tau,\Sig\in\T_{\mu^1+\delta}^3}\E_{\mu^1+\delta}\left[U^1(\Rho[\Tau,\Sig],\Tau[\Rho,\Sig],\Sig[\Rho,\Tau])\right],
\end{equation}
and $(\Tau^2, (\underline\Rho^2=(\underline\rho^2,\underline\rho_2^2,\underline\rho_3^2,\underline\rho_{23}^2),\underline\Sig^2=(\underline\sigma^2,\underline\sigma_1^2,\underline\sigma_2^2,\underline\sigma_{12}^2)))\in(\T_{\mu^2+\delta}^3)^3$ being an $\eps$-saddle point for the game
$$\sup_{\Tau\in\T_{\mu^2+\delta}^3}\inf_{\Rho,\Sig\in\T_{\mu^2+\delta}^3}\E_{\mu^2+\delta}\left[U^2(\Rho[\Tau,\Sig],\Tau[\Rho,\Sig],\Sig[\Rho,\Tau])\right],$$
and 
$(\Sig^3, (\underline\Rho^3=(\underline\rho^3,\underline\rho_2^3,\underline\rho_3^3,\underline\rho_{23}^3),\underline\Tau^3=(\underline\tau^3\underline\tau_1^3,\underline\tau_3^3,\underline\tau_{13}^3)))\in(\T_{\mu^3+\delta}^3)^3$ being an $\eps$-saddle point for the game
$$\sup_{\Sig\in\T_{\mu^3+\delta}^3}\inf_{\Rho,\Tau\in\T_{\mu^3+\delta}^3}\E_{\mu^3+\delta}\left[U^3(\Rho[\Tau,\Sig],\Tau[\Rho,\Sig],\Sig[\Rho,\Tau])\right].$$

Now define $\hat\Rho=(\hat\rho,\hat\rho_2,\hat\rho_3,\hat\rho_{23}), \hat\Tau=(\hat\tau,\hat\tau_1,\hat\tau_3,\hat\tau_{13}), \hat\Sig=(\hat\sigma,\hat\sigma_1,\hat\sigma_2,\hat\sigma_{12})$ as follows.
$$\hat\rho:=
\begin{cases}
\mu^1,&\text{on}\ A,\\
\underline\rho^2,&\text{on}\ B,\\
\underline\rho^3,&\text{on}\ C,
\end{cases}
\quad\quad\quad\quad\quad
\hat\rho_2(t):=
\begin{cases}
\underline\rho_2^2(t),&\text{on}\ B\cap\{t\geq\mu^2+\delta\},\\
\underline\rho_2^3(t),&\text{on}\ C\cap\{t\geq\mu^3+\delta\},\\
\overline\rho_{\phi_h(t)}^2,&\text{otherwise},
\end{cases}
$$
$$
\hat\rho_3(t):=
\begin{cases}
\underline\rho_3^2(t),&\text{on}\ B\cap\{t\geq\mu^2+\delta\},\\
\underline\rho_3^3(t),&\text{on}\ C\cap\{t\geq\mu^3+\delta\},\\
\overline\rho_{\phi_h(t)}^3,&\text{otherwise},
\end{cases}
\quad
\hat\rho_{23}(s,t):=
\begin{cases}
\underline\rho_{23}^2(s,t),&\text{on}\ B\cap\{\mu^2+\delta\leq s\wedge t\},\\
\underline\rho_{23}^3(s,t),&\text{on}\ C\cap\{\mu^3+\delta\leq s\wedge t\},\\
\underline\rho_{\phi_h(t)}^{3,23},&\text{on}\ B\cap\{s=t=\mu^2\},\\
\underline\rho_{\phi_h(t)}^{2,23},&\text{on}\ C\cap\{s=t=\mu^3\},\\
\overline\rho_{\phi_h(s),3}^2(t),&\text{on}\ D\cap\{s\leq t\},\\
\overline\rho_{\phi_h(t),2}^3(s),&\text{on}\ D\cap\{s>t\},\\
\end{cases}
$$
and
$$\hat\tau:=
\begin{cases}
\mu^2,&\text{on}\ B,\\
\underline\tau^1,&\text{on}\ A,\\
\underline\tau^3,&\text{on}\ C,
\end{cases}
\quad\quad\quad\quad\quad
\hat\tau_1(t):=
\begin{cases}
\underline\tau_1^1(t),&\text{on}\ A\cap\{t\geq\mu^1+\delta\},\\
\underline\tau_1^3(t),&\text{on}\ C\cap\{t\geq\mu^3+\delta\},\\
\overline\tau_{\phi_h(t)}^1,&\text{otherwise},
\end{cases}
$$
$$
\hat\tau_3(t):=
\begin{cases}
\underline\tau_3^1(t),&\text{on}\ A\cap\{t\geq\mu^1+\delta\},\\
\underline\tau_3^3(t),&\text{on}\ C\cap\{t\geq\mu^3+\delta\},\\
\overline\tau_{\phi_h(t)}^3,&\text{otherwise},
\end{cases}
\quad
\hat\tau_{13}(s,t):=
\begin{cases}
\underline\tau_{13}^1(s,t),&\text{on}\ A\cap\{\mu^1+\delta\leq s\wedge t\},\\
\underline\tau_{13}^3(s,t),&\text{on}\ C\cap\{\mu^3+\delta\leq s\wedge t\},\\
\underline\tau_{\phi_h(t)}^{3,13},&\text{on}\ A\cap\{s=t=\mu^1\},\\
\underline\tau_{\phi_h(t)}^{1,13},&\text{on}\ C\cap\{s=t=\mu^3\},\\
\overline\tau_{\phi_h(s),3}^1(t),&\text{on}\ E\cap\{s\leq t\},\\
\overline\tau_{\phi_h(t),1}^3(s),&\text{on}\ E\cap\{s>t\},\\
\end{cases}
$$
and
$$\hat\sigma:=
\begin{cases}
\mu^3,&\text{on}\ C,\\
\underline\sigma^1,&\text{on}\ A,\\
\underline\sigma^2,&\text{on}\ B,
\end{cases}
\quad\quad\quad\quad\quad
\hat\sigma_1(t):=
\begin{cases}
\underline\sigma_1^1(t),&\text{on}\ A\cap\{t\geq\mu^1+\delta\},\\
\underline\sigma_1^2(t),&\text{on}\ B\cap\{t\geq\mu^2+\delta\},\\
\overline\sigma_{\phi_h(t)}^1,&\text{otherwise},
\end{cases}
$$
$$
\hat\sigma_2(t):=
\begin{cases}
\underline\sigma_2^1(t),&\text{on}\ A\cap\{t\geq\mu^1+\delta\},\\
\underline\sigma_2^2(t),&\text{on}\ B\cap\{t\geq\mu^2+\delta\},\\
\overline\sigma_{\phi_h(t)}^2,&\text{otherwise},
\end{cases}
\quad
\hat\sigma_{12}(s,t):=
\begin{cases}
\underline\sigma_{12}^1(s,t),&\text{on}\ A\cap\{\mu^1+\delta\leq s\wedge t\},\\
\underline\sigma_{12}^2(s,t),&\text{on}\ B\cap\{\mu^2+\delta\leq s\wedge t\},\\
\underline\sigma_{\phi_h(t)}^{2,12},&\text{on}\ A\cap\{s=t=\mu^1\},\\
\underline\sigma_{\phi_h(t)}^{1,12},&\text{on}\ B\cap\{s=t=\mu^2\},\\
\overline\sigma_{\phi_h(s),2}^1(t),&\text{on}\ F\cap\{s\leq t\},\\
\overline\sigma_{\phi_h(t),1}^2(s),&\text{on}\ F\cap\{s>t\},\\
\end{cases}
$$

where
$$A:=\{\mu^1\leq\mu^2,\mu^1\leq\mu^3\},\quad B:=\{\mu^2<\mu^1,\mu^2\leq\mu^3\},\quad C:=\{\mu^3<\mu^1,\mu^3<\mu^2\},$$
and $D$ (resp. $E,F$) is the complement of the first four cases in the definition of $\hat\rho_{23}$ (resp. $\hat\tau_{13},\hat\sigma_{12}$). It can be shown that $\hat\Rho,\hat\Tau,\hat\Sig\in\T_\theta^3$. The next result shows that $(\hat\Rho,\hat\Tau,\hat\Sig)$ is a $13\eps$-Nash equilibrium for the game \eqref{e1} when $N=3$. In particular, \thref{t1} holds for $N=3$.

\begin{proposition}
$(\hat\Rho,\hat\Tau,\hat\Sig)\in(\T_\theta^3)^3$ is a $13\eps$-Nash equilibrium for the game \eqref{e1} for $N=3$.
\end{proposition}

\begin{proof}
First, we have that
\begin{eqnarray}
\notag \E_\theta\left[U^1(\hat\Rho[\hat\Tau,\hat\Sig],\hat\Tau[\hat\Rho,\hat\Sig],\hat\Sig[\hat\Rho,\hat\Tau])\right]&=&\E_\theta\left[\E_{\mu^1}\left[U^1\left(\mu^1,\overline\Tau_{\phi_h(\mu^1)}^1\left[\overline\Sig_{\phi_h(\mu^1)}^1\right],\overline\Sig_{\phi_h(\mu^1)}^1\left[\overline\Tau_{\phi_h(\mu^1)}^1\right]\right)\right]1_A\right]\\
\notag &+&\E_\theta\left[\E_{\mu^2}\left[U^1\left(\overline\Rho_{\phi_h(\mu^2)}^2\left[\overline\Sig_{\phi_h(\mu^2)}^2\right],\mu^2,\overline\Sig_{\phi_h(\mu^2)}^2\left[\overline\Rho_{\phi_h(\mu^2)}^2\right]\right)\right]1_B\right]\\
\notag &+&\E_\theta\left[\E_{\mu^3}\left[U^1\left(\overline\Rho_{\phi_h(\mu^3)}^3\left[\overline\Tau_{\phi_h(\mu^3)}^3\right],\overline\Tau_{\phi_h(\mu^3)}^3\left[\overline\Rho_{\phi_h(\mu^3)}^3\right],\mu^3\right)\right]1_C\right]\\
\notag &=&\E_\theta\left[Z_{\mu^1}^1 1_A+Y_{\mu^2}^{1,2} 1_B+Y_{\mu^3}^{1,3} 1_C\right].
\end{eqnarray}

Now take $\Rho=(\rho,\rho_2,\rho_3,\rho_{23})\in\T_\theta^3$. We will show that
\begin{equation}\label{e20}
\E_\theta\left[U^1(\Rho[\hat\Tau,\hat\Sig],\hat\Tau[\Rho,\hat\Sig],\hat\Sig[\Rho,\hat\Tau])\right]\leq\E_\theta\left[U^1(\hat\Rho[\hat\Tau,\hat\Sig],\hat\Tau[\hat\Rho,\hat\Sig],\hat\Sig[\hat\Rho,\hat\Tau])\right]+13\eps.
\end{equation}
We consider seven cases.

\textbf{Case 1}: $D_1:=\{\rho<\mu^1\wedge\mu^2\wedge\mu^3\}$. As $X_t^1\leq Z_t^1$ for any $t\geq 0$,
$$\mu^1\leq\inf\{t\geq\theta:\ V_t^1\leq X_t^1+\eps\}.$$
Therefore, the process $(V_t)_{t\geq 0}$ is a sub-martingale from $\theta$ to $\mu^1$. Hence,
\begin{eqnarray}
\notag \E_\theta\left[U^1(\Rho[\hat\Tau,\hat\Sig],\hat\Tau[\Rho,\hat\Sig],\hat\Sig[\Rho,\hat\Tau])1_{D_1}\right]&=&\E_\theta\left[\E_\rho\left[U^1\left(\rho,\overline\Tau_{\phi_h(\rho)}^1\left[\overline\Sig_{\phi_h(\rho)}^1\right],\overline\Sig_{\phi_h(\rho)}^1\left[\overline\Tau_{\phi_h(\rho)}^1\right]\right)\right]1_{D_1}\right]\\
\notag &=&\E_\theta\left[Z_\rho^1 1_{D_1}\right]\\
\notag &\leq&\E_\theta\left[V_\rho^1 1_{D_1}\right]\\
\notag &=&\E_\theta\left[\E_{\mu^1\wedge\mu^2\wedge\mu^3}\left[V_{\rho\wedge\mu^1\wedge\mu^2\wedge\mu^3}^1\right] 1_{D_1}\right]\\
\notag &\leq&\E_\theta\left[V_{\mu^1\wedge\mu^2\wedge\mu^3}^1 1_{D_1}\right]\\
\notag &=&\E_\theta\left[\left(V_{\mu^1}^1 1_A+V_{\mu^2}^1 1_B+V_{\mu^3}^1 1_C\right) 1_{D_1}\right]\\
\notag &\leq&\E_\theta\left[\left(Z_{\mu^1}^1 1_A+Y_{\mu^2}^{1,2} 1_B+Y_{\mu^3}^{1,3} 1_C\right) 1_{D_1}\right]+3\eps.
\end{eqnarray}

\textbf{Case 2}: $D_2:=\{\mu^1\wedge\mu^2\wedge\mu^3\leq\rho<\mu^1\wedge\mu^2\wedge\mu^3+\delta\}\cap A$. By \cite[Lemma 2.15, page 8]{KS}, $\{\mu^1\wedge\mu^2\wedge\mu^3\leq\rho\}\cap A\in\cF_\rho$, and thus $D_2\in\cF_\rho$. Then by \eqref{e18},
\begin{eqnarray}
\notag \E_\theta\left[U^1(\Rho[\hat\Tau,\hat\Sig],\hat\Tau[\Rho,\hat\Sig],\hat\Sig[\Rho,\hat\Tau])1_{D_2}\right]&=&\E_\theta\left[\E_\rho\left[U^1\left(\rho,\overline\Tau_{\phi_h(\rho)}^1\left[\overline\Sig_{\phi_h(\rho)}^1\right],\overline\Sig_{\phi_h(\rho)}^1\left[\overline\Tau_{\phi_h(\rho)}^1\right]\right)\right]1_{D_2}\right]\\
\notag &=&\E_\theta\left[Z_\rho^1 1_{D_2}\right]\\
\notag &\leq&\E_\theta\left[Z_{\mu^1}^1 1_{D_2}\right]+\eps\\
\notag &=&\E_\theta\left[\left(Z_{\mu^1}^1 1_A+Y_{\mu^2}^{1,2} 1_B+Y_{\mu^3}^{1,3} 1_C\right) 1_{D_2}\right]+\eps.
\end{eqnarray} 

\textbf{Case 3}: $D_3:=\{\rho\geq\mu^1\wedge\mu^2\wedge\mu^3+\delta\}\cap A$. Recall $\tilde V_{\mu^1+\delta}$ defined in \eqref{e19}. By \prref{p1},
$$\tilde V_{\mu^1+\delta}=\sup_{\rho\in\cT_{\mu^1+\delta}}\inf_{\lambda\in\cT_{\mu^1+\delta}}\E_{\mu+\delta}\left[X_\rho^1 1_{\{\rho\leq\lambda\}}+\tilde Y_\lambda^1 1_{\{\rho>\lambda\}}\right],$$
where
$$\tilde Y_t^1:=\left(\sup_{\Rho\in\T_t^2}\inf_{\Sig\in\T_t^2}\E_t\left[U^1(\Rho[\Sig],t,\Sig[\Rho])\right]\right)\wedge\left(\sup_{\Rho\in\T_t^2}\inf_{\Tau\in\T_t^2}\E_t\left[U^1(\Rho[\Tau],\Tau[\Rho],t)\right]\right).$$
We have that
$$\sup_{\Rho\in\T_t^2}\inf_{\Sig\in\T_t^2}\E_t\left[U^1(\Rho[\Sig],t,\Sig[\Rho])\right]\leq\sup_{\Rho\in\T_t^2}\E_t\left[U^1\left(\Rho\left[\overline\Sig_{\phi_h(t)}^2\right],t,\overline\Sig_{\phi_h(t)}^2\left[\Rho\right]\right)\right]\leq Y_t^{1,2}+\eps,$$
and similarly $\sup_{\Rho\in\T_t^2}\inf_{\Tau\in\T_t^2}\E_t\left[U^1(\Rho[\Tau],\Tau[\Rho],t)\right]\leq Y_t^{1,3}+\eps$. Therefore, for any $t\geq 0$, $\tilde Y_t^1\leq Y_t^1$, and thus
$$\tilde V_{\mu^1+\delta}^1\leq V_{\mu^1+\delta}^1.$$
Then by \eqref{e18},
\begin{eqnarray}
\notag \E_\theta\left[U^1(\Rho[\hat\Tau,\hat\Sig],\hat\Tau[\Rho,\hat\Sig],\hat\Sig[\Rho,\hat\Tau])1_{D_3}\right]&=&\E_\theta\left[\E_{\mu^1+\delta}\left[U^1\left(\Rho\left[\underline\Tau^1,\underline\Sig^1\right],\underline\Tau^1\left[\Rho,\underline\Sig^1\right],\underline\Sig^1\left[\Rho,\underline\Tau^1\right]\right)\right]1_{D_3}\right]\\
\notag &\leq&\E_\theta\left[\tilde V_{\mu^1+\delta}^1 1_{D_3}\right]+\eps\\
\notag &\leq&\E_\theta\left[V_{\mu^1+\delta}^1 1_{D_3}\right]+\eps\\
\notag &\leq&\E_\theta\left[V_{\mu^1}^1 1_{D_3}\right]+2\eps\\
\notag &\leq&\E_\theta\left[Z_{\mu^1}^1 1_{D_3}\right]+3\eps\\
\notag &=&\E_\theta\left[\left(Z_{\mu^1}^1 1_A+Y_{\mu^2}^{1,2} 1_B+Y_{\mu^3}^{1,3} 1_C\right) 1_{D_3}\right]+3\eps.
\end{eqnarray}

\textbf{Case 4}: $D_4:=\{\rho=\mu^1\wedge\mu^2\wedge\mu^3\}\cap B$. It can be shown that
$$Y_t^{1,2}\geq\inf_{\sigma\in\cT_t}\E_t\left[U^1(t,t,\sigma)\right]-\eps\geq\E_t\left[U^1\left(t,t,\underline\sigma_{\phi_h(t)}^{1,12}\right)\right]-2\eps.$$
Therefore,
\begin{eqnarray}
\notag \E_\theta\left[U^1(\Rho[\hat\Tau,\hat\Sig],\hat\Tau[\Rho,\hat\Sig],\hat\Sig[\Rho,\hat\Tau])1_{D_4}\right]&=&\E_\theta\left[\E_{\mu^2}\left[U^1\left(\mu^2,\mu^2,\underline\sigma_{\phi_h(\mu^2)}^{1,12}\right)\right]1_{D_4}\right]\\
\notag &\leq&\E_\theta\left[Y_{\mu^2}^{1,2} 1_{D_4}\right]+2\eps\\
\notag &=&\E_\theta\left[\left(Z_{\mu^1}^1 1_A+Y_{\mu^2}^{1,2} 1_B+Y_{\mu^3}^{1,3} 1_C\right) 1_{D_4}\right]+2\eps.
\end{eqnarray}

\textbf{Case 5}: $D_5:=\{\rho>\mu^1\wedge\mu^2\wedge\mu^3\}\cap B$. Let $\Rho_{\mu^2}:=(\rho_2(\mu^2),\rho_{23}(\mu^2,\cdot))\in\T_{\mu^2}^2$.
\begin{eqnarray}
\notag \E_\theta\left[U^1(\Rho[\hat\Tau,\hat\Sig],\hat\Tau[\Rho,\hat\Sig],\hat\Sig[\Rho,\hat\Tau])1_{D_5}\right]&=&\E_\theta\left[\E_{\mu^2}\left[U^1\left(\Rho_{\mu^2}\left[\overline\Sig_{\phi_h(\mu^2)}^2\right],\mu^2,\overline\Sig_{\phi_h(\mu^2)}^2\left[\Rho_{\mu^2}\right]\right)\right]1_{D_5}\right]\\
\notag &\leq&\E_\theta\left[Y_{\mu^2}^{1,2} 1_{D_5}\right]+\eps\\
\notag &=&\E_\theta\left[\left(Z_{\mu^1}^1 1_A+Y_{\mu^2}^{1,2} 1_B+Y_{\mu^3}^{1,3} 1_C\right) 1_{D_5}\right]+\eps.
\end{eqnarray}

\textbf{Case 6}: $D_6:=\{\rho=\mu^1\wedge\mu^2\wedge\mu^3\}\cap C$. Similar to case 4, we can show that
$$\E_\theta\left[U^1(\Rho[\hat\Tau,\hat\Sig],\hat\Tau[\Rho,\hat\Sig],\hat\Sig[\Rho,\hat\Tau])1_{D_6}\right]\leq\E_\theta\left[\left(Z_{\mu^1}^1 1_A+Y_{\mu^2}^{1,2} 1_B+Y_{\mu^3}^{1,3} 1_C\right) 1_{D_6}\right]+2\eps.$$

\textbf{Case 7}: $D_7:=\{\rho>\mu^1\wedge\mu^2\wedge\mu^3\}\cap C$. Similar to case 5, we can show that
$$\E_\theta\left[U^1(\Rho[\hat\Tau,\hat\Sig],\hat\Tau[\Rho,\hat\Sig],\hat\Sig[\Rho,\hat\Tau])1_{D_7}\right]\leq\E_\theta\left[\left(Z_{\mu^1}^1 1_A+Y_{\mu^2}^{1,2} 1_B+Y_{\mu^3}^{1,3} 1_C\right) 1_{D_7}\right]+\eps.$$
From cases 1-7, we have \eqref{e20} holds. Similarly, we can show that for any $\Tau,\Sig\in\T_\theta^3$,
$$\E_\theta\left[U^2(\hat\Rho[\Tau,\hat\Sig],\Tau[\hat\Rho,\hat\Sig],\hat\Sig[\hat\Rho,\Tau])\right]\leq\E_\theta\left[U^2(\hat\Rho[\hat\Tau,\hat\Sig],\hat\Tau[\hat\Rho,\hat\Sig],\hat\Sig[\hat\Rho,\hat\Tau])\right]+13\eps,$$
and
$$\E_\theta\left[U^3(\hat\Rho[\hat\Tau,\Sig],\hat\Tau[\hat\Rho,\Sig],\Sig[\hat\Rho,\hat\Tau])\right]\leq\E_\theta\left[U^3(\hat\Rho[\hat\Tau,\hat\Sig],\hat\Tau[\hat\Rho,\hat\Sig],\hat\Sig[\hat\Rho,\hat\Tau])\right]+13\eps.$$ 
\end{proof}

\bibliographystyle{siam}
\bibliography{ref}

\begin{thebibliography}{10}

\bibitem{ZZ7}
{\sc E.~{Bayraktar} and Z.~{Zhou}}, {\em {On a Stopping Game in continuous
  time}}, ArXiv e-prints,  (2014).
\newblock To appear in the Proceedings of the American Mathematical Society.

\bibitem{ZZ6}
{\sc E.~Bayraktar and Z.~Zhou}, {\em On zero-sum optimal stopping games},
  (2014).
\newblock Preprint, arXiv:1408.3692.

\bibitem{Bismut}
{\sc J.-M. Bismut}, {\em Sur un probl\`eme de {D}ynkin}, Z.
  Wahrscheinlichkeitstheorie und Verw. Gebiete, 39 (1977), pp.~31--53.

\bibitem{Dynkin}
{\sc E.~B. Dynkin}, {\em A game-theoretic version of an optimal stopping
  problem}, Dokl. Akad. Nauk SSSR, 185 (1969), pp.~16--19.

\bibitem{Ferenstein}
{\sc E.~Z. Ferenstein}, {\em On randomized stopping games}, in Advances in
  dynamic games, vol.~7 of Ann. Internat. Soc. Dynam. Games, Birkh\"auser
  Boston, Boston, MA, 2005, pp.~223--233.

\bibitem{Hamadene}
{\sc S.~Hamad{\`e}ne}, {\em Mixed zero-sum stochastic differential game and
  {A}merican game options}, SIAM J. Control Optim., 45 (2006), pp.~496--518.

\bibitem{Zhang3}
{\sc S.~Hamad{\`e}ne and J.~Zhang}, {\em The continuous time nonzero-sum
  {D}ynkin game problem and application in game options}, SIAM J. Control
  Optim., 48 (2009/10), pp.~3659--3669.

\bibitem{KS}
{\sc I.~Karatzas and S.~E. Shreve}, {\em Brownian motion and stochastic
  calculus}, vol.~113 of Graduate Texts in Mathematics, Springer-Verlag, New
  York, 1988.

\bibitem{Kifer}
{\sc Y.~Kifer}, {\em Game options}, Finance Stoch., 4 (2000), pp.~443--463.

\bibitem{Ko3}
{\sc M.~Kobylanski, M.-C. Quenez, and M.~R. de~Campagnolle}, {\em Dynkin games
  in a general framework}, Stochastics, 86 (2014), pp.~304--329.

\bibitem{Solan1}
{\sc R.~Laraki and E.~Solan}, {\em The value of zero-sum stopping games in
  continuous time}, SIAM J. Control Optim., 43 (2005), pp.~1913--1922
  (electronic).

\bibitem{Solan}
\leavevmode\vrule height 2pt depth -1.6pt width 23pt, {\em Equilibrium in
  two-player non-zero-sum {D}ynkin games in continuous time}, Stochastics, 85
  (2013), pp.~997--1014.

\bibitem{Solan4}
{\sc R.~{Laraki}, E.~{Solan}, and N.~{Vieille}}, {\em {Continuous-time games of
  timing.}}, {J. Econ. Theory}, 120 (2005), pp.~206--238.

\bibitem{Lepeltier}
{\sc J.~{Lepeltier} and M.~{Maingueneau}}, {\em {Le jeu de Dynkin en theorie
  generale sans l'hypoth\`ese de Mokobodski.}}, {Stochastics}, 13 (1984),
  pp.~25--44.

\bibitem{Neveu}
{\sc J.~Neveu}, {\em Discrete-parameter martingales}, North-Holland Publishing
  Co., Amsterdam-Oxford; American Elsevier Publishing Co., Inc., New York,
  revised~ed., 1975.
\newblock Translated from the French by T. P. Speed, North-Holland Mathematical
  Library, Vol. 10.

\bibitem{Solan3}
{\sc D.~{Rosenberg}, E.~{Solan}, and N.~{Vieille}}, {\em {Stopping games with
  randomized strategies.}}, {Probab. Theory Relat. Fields}, 119 (2001),
  pp.~433--451.

\bibitem{Solan2}
{\sc E.~Shmaya and E.~Solan}, {\em Two-player nonzero-sum stopping games in
  discrete time}, Ann. Probab., 32 (2004), pp.~2733--2764.

\bibitem{Touzi}
{\sc N.~{Touzi} and N.~{Vieille}}, {\em {Continuous-time Dynkin games with
  mixed strategies.}}, {SIAM J. Control Optim.}, 41 (2002), pp.~1073--1088.

\bibitem{ZZ9}
{\sc Z.~{Zhou}}, {\em {Non-zero-sum stopping games in continuous time}}, ArXiv
  e-prints,  (2015).

\bibitem{ZZ10}
\leavevmode\vrule height 2pt depth -1.6pt width 23pt, {\em {Non-zero-sum
  stopping games in discrete time}}, ArXiv e-prints,  (2015).

\end{thebibliography}

\end{document}